\documentclass{article}
\usepackage{graphicx} 
\usepackage[utf8]{inputenc}
\usepackage{amsmath}
\usepackage{amssymb}
\usepackage{textcomp}
\usepackage[english]{babel}
\usepackage{amsthm}
\usepackage{hyperref}
\usepackage{tikz}
\usetikzlibrary{arrows,decorations.pathmorphing,backgrounds,positioning,fit,matrix}
\theoremstyle{definition}
\newtheorem{definition}{Definition}[section]

\newtheorem{theorem}{Theorem}[section]

\newtheorem{proposition}[theorem]{Proposition}

\newtheorem{claim}[theorem]{Claim}

\newtheorem{fact}[theorem]{Fact}

\newtheorem*{acknowledgements}{Acknowledgements}
\title{A note on
Erd\H{o}s-Hajnal property for graphs with VC dimension $\leq 2$}
\author{Yayi Fu}
\date{}
\begin{document}
\maketitle
\begin{abstract}
Using techniques in \cite{chudnovsky2023erdHos} and substitution in \cite{alon2001ramsey}, we show that there is $\epsilon>0$ such that for any graph $G$ with VC-dimension $\leq 2$, $G$ has a clique or an anti-clique of size $\geq |G|^\epsilon$.
We also show that Erd\H os-Hajnal
property of VC-dimension $1$ 
graphs can be 
proved using 
$\delta$-dimension technique in \cite{chernikov2018note},
and we show that when $E$ is a
definable symmetric binary relation,
\cite[Theorem~1.3]{chernikov2018note}
can be proved without using Shelah’s 2-rank.
\end{abstract}
\section{Introduction}
\indent 

\emph{Erd\H{o}s-Hajnal conjecture} \cite{erdos1989ramsey} says for any graph $H$ there is $\epsilon>0$ such that if a graph $G$ does not contain any induced subgraph isomorphic to $H$ then $G$ has a clique or an  anti-clique of size $\geq |G|^\epsilon$.
More generally, we say a family of finite graphs has the \emph{Erd\H{o}s-Hajnal property} if there is $\epsilon>0$ such that for any graph $G$ in the family, $G$ has a clique or an anti-clique of size $\geq|G|^\epsilon$. 
Malliaris and Shelah proved in \cite{malliaris2014regularity} that the family of stable graphs has the Erd\H{o}s-Hajnal property. 
Chernikov and Starchenko gave another proof for stable graphs in \cite{chernikov2018note} and in \cite{chernikov2018regularity} they proved that the family of distal graphs has the strong Erd\H{o}s-Hajnal property. 
In general, we are interested in whether the family of finite VC-dimension (i.e. NIP \cite{simon2015guide}) graphs, which contains both stable graphs and distal graphs, has the Erd\H{o}s-Hajnal property. 
Motivation for studying this problem was given in \cite{fox2019erdHos}, which also gave a lower bound $e^{(\log n)^{1-o(1)}}$ for largest clique or anti-clique in a graph with bounded VC dimension. 
In this paper, we will show Erd\H os-Hajnal property for graphs with VC-dimension $\leq 2$.
\\
\indent
Section \ref{prelim} gives basic settings of graphs, stability, VC-dimension, ultraproduct, $\delta$-dimension.
\\
\indent
Section \ref{vc1} shows we can use the same technique 
in \cite{chernikov2018note} to show the Erd\H os-Hajnal property for graphs with VC-dimension $1$, 
which was proved using combinatorics.
\begin{theorem}
 The family of graphs with VC-dimension $\leq 1$ has the Erd\H{o}s-Hajnal property.
\end{theorem}
Section \ref{stabcase} shows 
Erd\H os-Hajnal property for stable graphs
can be proved without using Shelah’s 2-rank.
\begin{theorem}
 For each $k\in\mathbb{N}$,
 the family of $k$-stable graphs has the Erd\H{o}s-Hajnal property.    
\end{theorem}
Section \ref{vcq2} shows that Erd\H os-Hajnal property holds for graphs with VC-dimension $2$. 
\begin{theorem}
    The family of graphs with VC-dimension $\leq 2$ has the Erd\H{o}s-Hajnal property.
\end{theorem} 
\begin{acknowledgements}
The author is grateful to her advisor Sergei Starchenko for helpful suggestions.
\end{acknowledgements}
\section{Preliminaries}\label{prelim}
\indent

A \emph{graph} $G$ is a structure $(V,E)$ where $V$ is the underlying set 
($V$ can be finite or infinite),
$E$ is a symmetric anti-reflexive binary relation.
$H$ is an \emph{induced subgraph} of $G$
if $H\subseteq G$ as a substructure. 
$G$ is \emph{$H$-free} if $G$ does not 
contain $H$ as an induced subgraph. 
If $\mathcal{H}$ is a family of graphs, 
we say $G$ is $\mathcal{H}$-free if for any 
$H\in\mathcal{H}$, $G$ is $H$-free. 
$\overline{G}$ denotes the \emph{complement} of $G$, 
i.e. $G=(V,E)$ and $\overline{G}=(V,\overline{E})$ have the same vertex set 
$V$ and for any distinct vertices $a,b\in V$,
$a\overline{E}b$ in $\overline{G}$ iff 
$\neg aEb$ in $G$.
A subset $A\subseteq V$ is a 
\emph{homogeneous set} if
the induced subgraph $A$ is a clique
or an anti-clique.
\\
\indent
For $a\in V$, $A\subseteq V$,
let $E(a,A)$ denote the set
$\{x\in A: 
E(a,x)\}$ and
let $\neg
E(a,A)$ denote the set
$\{x\in A:
\neg
E(a,x)\}$.
\\
\indent
We use the pseudo-finite setting in \cite{chernikov2018note}:
\\
\indent
Let $\{G_i=(V_i,E_i):
i\in\omega\}$ be a sequence of 
finite graphs.
Let $\mathcal{F}$ be a non-principal ultrafilter of $\omega$.
Let $G=(V,E)$ be the ultraproduct
$\underset{i\in\omega}{\prod}
(V_i,E_i)/
\mathcal{F}$. 
(For simplicity,
we write it as
$\underset{i\in\omega}{\prod}
V_i/
\mathcal{F}$.)
\\
\indent
Let $A$ be an internal set 
$\underset{i\in\omega}{\prod}
A_i/\mathcal{F}$, where each $A_i$ is a non-empty subset of $V_i$.
For each $i\in \omega$, 
let $l_i = \log(|A_i|)/ \log(|V_i|)$.
We define the \emph{$\delta$-dimension of $A$}, denoted by $\delta(A)$, to be the unique number $l\in[0,1]$
such that for any
$\epsilon\in\mathbb{R}^{>0}$,
the set $\{i\in\omega:
l-\epsilon<l_i <l+\epsilon\}$ is in $\mathcal{F}$. 
\begin{definition}
    Let $G=(V,E)$ be a graph.
    Let $k\in\mathbb{N}$.
    $G$ is \emph{$k$-stable} if there do not exist some $a_1,...,a_k\in V$,
    $b_1,...,b_k
    \in V$ such that 
    $E(a_i,b_j)$ holds if and only if $i\leq j$. 
\end{definition}
\begin{fact}\label{stabequiv}
    \cite[Theorem~2.2]{shelah1990classification}
    Let $G=(V,E)$  be the ultraproduct
$\underset{i\in\omega}{\prod}V_i/
\mathcal{F}$. 
$G$ is unstable for all $k\in\mathbb{N}$ iff
there is $A\subseteq V$ and $\lambda\geq\aleph_0$ such that $|S^1_E(A)| > \lambda \geq|A|$.
($S^1_E(A):=
\{\underset{a\in A}{\bigcap} E(x;a)^{\epsilon(\Bar{a})}:
\epsilon\in2^{A}\}$.)
\end{fact}
\begin{definition}
For $d\in\mathbb{N}$,
a graph $G=(V,E)$ is of VC-dimension $<d$ if
there is no $d$-tuple $(x_0,...,x_{d-1})$
of pairwise distinct vertices in $V$ such that for all $\epsilon\in 2^d$,
there is $a_{\epsilon}\in V$ such that 
$\underset{i\in d}{\bigwedge}
E(a_\epsilon,x_i)^{\epsilon(i)}$.
\end{definition}
\begin{fact}\label{c6eh}
    \cite[1.9]{chudnovsky2023erdHos}
    The family $\{G: G $ is $
 \{C_6,\overline{C}_6\}$-free $\}$ has Erd\H os-Hajnal property. ($C_6$ is the $6$-cycle. )
\end{fact}
\section{VC-dimension 1}\label{vc1}
\begin{definition}
  Let $G=(V,E)$ be the ultraproduct
$\underset{i\in\omega}{\prod}V_i/
\mathcal{F}$. 
For a definable set $A\subseteq V$ such that $\delta(A)>0$,
we say that $A$ satisfies \emph{Property $(*)$} if there is a definable 
$A^+ \subseteq \{a\in A\, |\, \delta (\{x\in A \,
|\,  E(x,a)\})<
\delta (A)\}$
such that $\delta (A^+)=\delta (A)$ or there 
is a definable $A^- \subseteq \{a\in A\, |\, \delta (\{x\in A \, |\,  \neg E(x,a)\})
<\delta (A)\}$ such that $\delta (A^-)=\delta (A)$.
\\
\indent
         For a definable subset $S$ and a vertex $s\in S$, we say that $\emph{s splits S}$ if $\delta(\{x\in S\,|\,E(x,s)\})>0$ and $\delta(\{x\in S\,|\,\neg E(x,s)\})>0$.
\end{definition}
\begin{proposition}\label{prop*}
 Let $G=(V,E)$ be the ultraproduct
$\underset{i\in\omega}{\prod}V_i/
\mathcal{F}$. 
Assume $A\subseteq V$ is definable with $\delta (A)>0$, and $A$ satisfies property $(*)$.
Then $A$ has a homogeneous subset with positive $\delta$-dimension.
\end{proposition}
\begin{proof}
Let $A\subseteq V$ be definable with $\delta (A)>0$,
and $A$ satisfies property $(*)$. 
May assume that there is a definable 
 $A^+ \subseteq \{a\in A\, |\, \delta (\{x\in A \, |\,  E(x,a)\})
 < \delta (A)\}$
such that $\delta (A^+)=\delta (A)$.
\\
\indent
For $a\in A^+$, $\delta (\{x\in A^+\, |\, E(x,a)\})\leq \delta(\{ x\in A \, |\, E(x,a)\})< \delta (A)=
\delta(A^+)$.
\begin{claim}\label{unifbdd}
    Suppose $A\subseteq V$ is definable and there is $\alpha>0$ such that for all $a\in A$, $\delta(E(a,A))<\alpha$. Then there is $\beta<\alpha$ such that
    for all $a\in A$, 
    $\delta(E(a,A))\leq\beta$.
    \end{claim}
\begin{proof}
    Let $0<\alpha_1<\alpha_2<...$ be a sequence increasing to $\alpha$. 
    By adding relation symbols as in \cite{bays2018projective}, we may assume there exist $D_n$ definable such that\\
$\{y\in A\,|\,\delta(E(y,A))\geq\alpha_{n+1}\}\subseteq D_n\subseteq\{y\in A\,|\,\delta(E(y,A))\geq\alpha_n\}$. 
If all $D_n$'s are not empty, by $\omega_1$-saturation and compactness, $\underset{n}{\bigcap}D_n\neq\emptyset$. Then there is $a\in A$ such that $\delta(E(a,A))\geq\alpha$, a contradiction. So $D_n=\emptyset$ for some $n$.
\end{proof} 
Hence, by claim \ref{unifbdd},
there is $\epsilon\in(0,1)$ such that
 for all $a\in A^+$, 
    $\delta(E(a,A^+))\leq\epsilon\delta(A)$.
    \\
    \indent
(Similar to the proof in \cite{chernikov2018note}.) 
Let $A^+=\prod A_i/\mathcal{F}$. For each $i\in \omega$, let $B_i\subseteq A_i$ be maximal such that $\neg E_i(x,y)$ for all $x,y\in B_i$. 
Let $B=\prod B_i/\mathcal{F}$. 
Then 
\begin{enumerate}
    \item[(i)]
$B\subseteq A^+.$
\item[(ii)]
$V\vDash (\forall x,y\in B)$ $\neg E(x,y)$. \item[(iii)] 
For any $a\in A^+\setminus B$, there is $b\in B$ such that $V\vDash E(a,b)$.
\end{enumerate}

Hence, $A^+\setminus B\subseteq \underset{b\in B}{\bigcup} \{ x\in A^+\, |\, E(x,b) \}$ and $\delta(A^+\setminus B)\leq \delta (B)+
\epsilon\delta (A)
=
\delta (B)+
\epsilon\delta (A^+)$.
So $\delta (B)>0$.
\\
\indent
Proof is similar if there is a definable $A^- \subseteq \{a\in A\, |\, \delta (\{x\in A \, |\,  \neg E(x,a)\})
< 
\delta (A)\}$ such that 
$\delta (A^-)
=\delta (A)$.
\end{proof}   
\begin{claim}\label{claim1}
     Fix a definable $A$ such that $\delta (A)>0$. 
     Then the set $\{a\in A\,  |\, \delta(\{x\in 
     A\,|\, E(x,a)\})<\delta(A)\}$ is a countable 
     union of definable sets. 
     The same holds for $\{a\in A\,  |\, \delta(\{x\in A\,|\, \neg E(x,a)\})<\delta(A)\}$.
\end{claim}
\begin{proof}
$\{a\in A\,  |\, \delta(\{x\in A\,|\, E(x,a)\})<\delta(A)\}=
\underset{n\in\omega}{\bigcup}\{a\in A\,  |\, \delta(\{x\in A\,|\, E(x,a)\})<
\delta(A)-\frac{1}{n}\}$. 
By continuity of $\delta$-dimension, for each $n\in\omega$,
there is a definable $D_n$ such that $\{a\in A\,  |\, \delta(\{x\in A\,|\, E(x,a)\})
<\delta(A)-\frac{1}{n}\}\subseteq D_n\subseteq 
\\
\{a\in A\,  |\, \delta(\{x\in A\,|\, E(x,a)\})<\delta(A)-\frac{1}{n+1}\}$. 
\\
\indent
Hence, $\{a\in A\,  |\, \delta(\{x\in A\,|\, E(x,a)\})
<\delta(A)\}=\underset{n\in\omega}{\bigcup}D_n$.
\\
\indent
Similar for $\{a\in A\,  |\, \delta(\{x\in A\,|\, \neg E(x,a)\})
<\delta(A)\}$.
\end{proof}
\begin{claim}\label{claim2}
 Fix a definable $A$ such that $\delta(A)>0$. If property $(*)$ fails for $A$,
 i.e. if for all definable $B\subseteq A$ with $\delta(B)=\delta(A)$, 
 \\
 $B\nsubseteq \{a\in A\,|\,\delta(\{x\in A\,|\,E(x,a)\})
 <
 \delta (A)\}$ and 
 \\
 $B\nsubseteq \{a\in A\,|\,\delta(\{x\in A\,|\,\neg E(x,a)\})
 <
 \delta (A)\}$, 
 then for all $B\subseteq A$ with $\delta(B)=\delta(A)$, 
 \\
 $B\nsubseteq \{a\in A\,|\,\delta(\{x\in A\,|\,E(x,a)\})
 <
 \delta (A)\}\cup\{a\in A\,|\,\delta(\{x\in A\,|\,\neg E(x,a)\})
 <
 \delta (A)\}$.
 \\
 \indent
Moreover, suppose property $(*)$ fails for all $A$ with $\delta(A)>0$.
Fix $A$ with $\delta(A)>0$.
Then for any $B\subseteq A$ with $\delta(B)=\delta(A)$,
there exist $a,a'\in B$, $a\neq a'$ such that $\delta(\{x\in A\,|\, E(x,a)\})>0$,
$\delta(\{x\in A\,|\,\neg E(x,a)\})>0$, 
$\delta(\{x\in A\,|\, E(x,a')\})>0$, $\delta(\{x\in A\,|\,\neg E(x,a')\})>0$ and $E(a,a')$.
\end{claim}
\begin{proof}
     By Claim \ref{claim1}, let$\{a\in A\,  |\, \delta(\{x\in A\,|\, E(x,a)\})
     <
     \delta(A)\}=\underset{n\in\omega}{\bigcup}D_n$ and 
     $\{a\in A\,  |\, \delta(\{x\in A\,|\, \neg E(x,a)\})
     <\delta(A)\}=\underset{n\in\omega}{\bigcup}F_m$.
     Fix $B\subseteq A$ definable such that $\delta(B)=\delta(A)$. 
\\
\indent
Consider $\Sigma:=\{B(x)\}\cup\{\neg D_n(x), \neg F_m(x)\,|\, n<\omega, m<\omega\}$.
If $B\subseteq \underset{n\in\Delta}{\bigcup}D_n\cup\underset{m\in\Delta'}{\bigcup}F_m$ for some finite $\Delta$,
$\Delta'\subseteq\omega$,
then there is some $D_n$ (or $F_m$) such that $\delta(D_n)\geq\delta(B)$ (or $\delta(F_m)\geq\delta(
B)$), contradicting the assumption. By $\omega_1$-saturation of $V$, $\Sigma$ is realized in $V$, and we have the conclusion.
\\
\indent
For the moreover part, consider $\Sigma':=\{B(x), B(y), x\neq y, E(x,y)\}\cup\\\{\neg D_n(x), \neg D_n(y), \neg F_m(x), \neg F_m(y)\,|\, n<\omega, m<\omega\}$. By assumption, we have $\delta(B\setminus\underset{n\in\Delta}{\bigcup}D_n\cup\underset{m\in\Delta'}{\bigcup}F_m)=\delta(B)$. Then there exist $b_1\neq b_2$ in $B\setminus\underset{n\in\Delta}{\bigcup}D_n\cup\underset{m\in\Delta'}{\bigcup}F_m$ such that $E(b_1, b_2)$ (Otherwise, $B\setminus\underset{n\in\Delta}{\bigcup}D_n\cup\underset{m\in\Delta'}{\bigcup}F_m$ satisfies property $(*)$, contradicting the assumption that property $(*)$ fails for all sets with positive $\delta$-dimension). By compactness and $\omega_1$-saturation of $V$, $\Sigma'$ is realized in $V$.
\end{proof}
\begin{theorem}
 The family of finite graphs with VC-dimension $\leq 1$ has the Erd\H{o}s-Hajnal property.
\end{theorem}
\begin{proof}
Suppose no. 
For each $i\in\omega$,
let $G_i=(V_i,E_i)$ be a finite graph with
VC-dimension $\leq 1$ such that
all homogeneous subsets of $G_i$ has size 
$<|V_i|^{\frac{1}{i}}$.
 Let $G=(V,E)$ be the ultraproduct
$\underset{i\in\omega}{\prod}V_i/
\mathcal{F}$. 
By proposition \ref{prop*},
property $(*)$ fails for all $A\subseteq V$ with $\delta(A)>0$. By claim \ref{claim2}, there is $a_0\in V$ that splits $V$. Let $L:= \{x\in V\,|\,x\neq a_0\wedge\neg E(x,a_0)\}$ and  $R:=\{x\in V\,|\,x\neq a_0\wedge E(x,a_0)\}$.\\ 
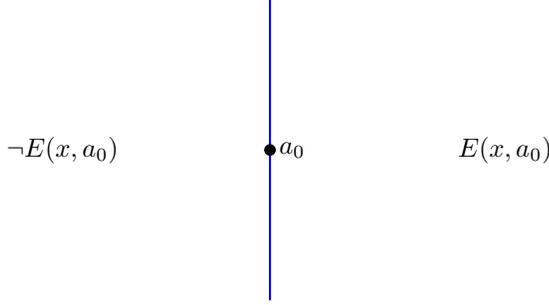
\begin{tikzpicture}
\draw[blue, thick] (0,2) -- (0,-2);
\filldraw[black] (0,0) circle (2pt) node[anchor=west] {$a_0$};
\node[text width=3cm] at (-2,0){$\neg E(x,a_0)$};
\node[text width=3cm] at (4,0){$E(x,a_0)$};
\end{tikzpicture}
\\
\indent
If there is $c\in R$ splitting $R$ such that there exist $d_1, d_2\in L$ with $E(c,d_1)\wedge \neg E(c,d_2)$,  then take $b_0=a_0$, $b_1=c$.\\ 
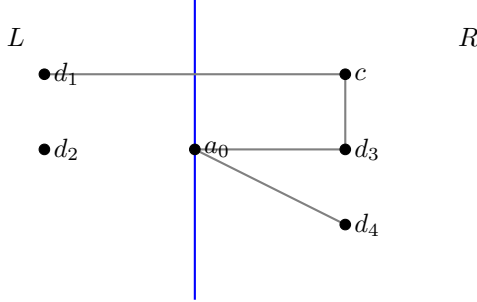
\begin{tikzpicture}
\draw[blue, thick] (0,2) -- (0,-2);
\draw[gray, thick] (-2,1) -- (2,1);
\draw[gray, thick] (2,1) -- (2,0);
\draw[gray, thick] (2,0) -- (0,0);
\draw[gray, thick] (2,-1) -- (0,0);
\filldraw[black] (0,0) circle (2pt) node[anchor=west] {$a_0$};
\filldraw[black] (-2,1) circle (2pt) node[anchor=west] {$d_1$};
\filldraw[black] (-2,0) circle (2pt) node[anchor=west] {$d_2$};
\filldraw[black] (2,0) circle (2pt) node[anchor=west] {$d_3$};
\filldraw[black] (2,-1) circle (2pt) node[anchor=west] {$d_4$};
\filldraw[black] (2,1) circle (2pt) node[anchor=west] {$c$};
\node[text width=3cm] at (-1,1.5){$L$};
\node[text width=3cm] at (5,1.5){$R$};
\end{tikzpicture}
\\
\indent
By the choice of $c$, there exist $d_3$, $d_4\in R$ such that $E(d_3,c)\wedge \neg E(d_4,c)$. Then $a_0$, $c$; $d_1$, $d_2$, $d_3$, $d_4$ witness that $E$ has VC-dimension $>1$, a contradiction.\\
Otherwise, assume for any $c\in R$ splitting $R$, we have for all $d\in L$, $E(d,c)$ or for all $d\in L$, $\neg E(d,c)$. (There is some $c\in R$ splitting $R$ by claim \ref{claim2}.)
\\
\indent
Suppose $c\in R$ splits $R$ and for all $d\in L$, $E(d,c)$. By claim \ref{claim2}, let $d_1\in L$ split $L$. We say $L_1=\{x\in L\,|\,x\neq d_1\wedge\neg E(x,d_1)\}$ and $R_1=\{x\in L\,|\, x\neq d_1\wedge E(x,d_1)\}$. If $\forall x\in L_1$ splitting $L_1$, $\forall y\in R_1$, $E(x,y)$, then take $d_2\in R_1$ such that $d_2$ splits $R_1$. Take $d_3\in R_1$ such that $\neg E(d_3,d_2)$. Take $d_4\in L_1$ splitting $L_1$. Then $E(d_4, d_3)\wedge\neg E(d_4, d_1)$. Thus, $d_1$, $d_3$; $a_0$, $c$, $d_2$, $d_4$ witness that $E$ has VC-dimension $>1$.\\
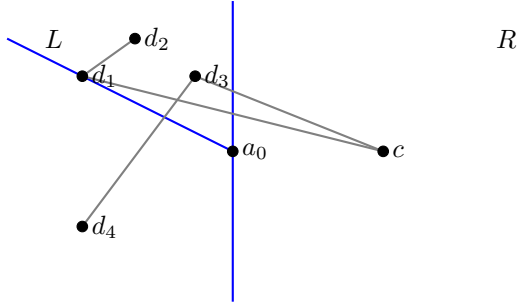
\begin{tikzpicture}
\draw[blue, thick] (0,2) -- (0,-2);
\draw[blue, thick] (-3,1.5) -- (0,0);
\draw[gray, thick] (-2,1) -- (-1.3,1.5);
\draw[gray, thick] (-2,1) -- (2,0);
\draw[gray, thick] (-0.5,1) -- (2,0);
\draw[gray, thick] (-2,-1) -- (-0.5,1);
\filldraw[black] (0,0) circle (2pt) node[anchor=west] {$a_0$};
\filldraw[black] (-2,1) circle (2pt) node[anchor=west] {$d_1$};
\filldraw[black] (-1.3,1.5) circle (2pt) node[anchor=west] {$d_2$};
\filldraw[black] (-0.5,1) circle (2pt) node[anchor=west] {$d_3$};
\filldraw[black] (-2,-1) circle (2pt) node[anchor=west] {$d_4$};
\filldraw[black] (2,0) circle (2pt) node[anchor=west] {$c$};
\node[text width=3cm] at (-1,1.5){$L$};
\node[text width=3cm] at (5,1.5){$R$};
\end{tikzpicture}
\\
\indent
On the other hand, if we have $\exists x\in L_1$ splitting $L_1$, $\exists y\in R_1$, $\neg E(x,y)$, take $d_2\in L_1, d_3\in R_1$ such that $d_2$ splits $L_1$ and $\neg E(d_2,d_3)$.Take $d_4\in L_1$ such that $E(d_2,d_4)$. Thus, $d_1$, $d_2$; $a_0$, $c$, $d_3$, $d_4$ witness that $E$ has VC-dimension $>1$. \\
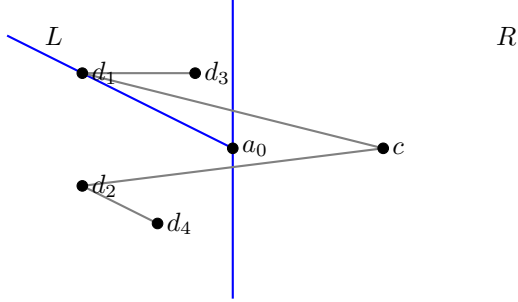
\begin{tikzpicture}
\draw[blue, thick] (0,2) -- (0,-2);
\draw[blue, thick] (-3,1.5) -- (0,0);
\draw[gray, thick] (-2,1) -- (-0.5,1);
\draw[gray, thick] (-2,1) -- (2,0);
\draw[gray, thick] (-2,-0.5) -- (2,0);
\draw[gray, thick] (-2,-0.5) -- (-1,-1);
\filldraw[black] (0,0) circle (2pt) node[anchor=west] {$a_0$};
\filldraw[black] (-2,1) circle (2pt) node[anchor=west] {$d_1$};
\filldraw[black] (-2,-0.5) circle (2pt) node[anchor=west] {$d_2$};
\filldraw[black] (-0.5,1) circle (2pt) node[anchor=west] {$d_3$};
\filldraw[black] (-1,-1) circle (2pt) node[anchor=west] {$d_4$};
\filldraw[black] (2,0) circle (2pt) node[anchor=west] {$c$};
\node[text width=3cm] at (-1,1.5){$L$};
\node[text width=3cm] at (5,1.5){$R$};
\end{tikzpicture}
\\
\indent
Hence for all $c\in R$ splitting $R$, we have for all $d\in L$, $\neg E(d,c)$. Take any $c_1\in R$ that splits $R$. We say $L_1=\{x\in R\,|\,x\neq c_1\wedge\neg E(x,c_1)\}$ and $R_1=\{x\in R\,|\, x\neq c_1\wedge E(x,c_1)\}$. If $\forall x\in L_1$ splitting $L_1$, $\forall y\in R_1$, $E(x,y)$, then take $c_2$ that splits $R_1$ and $c_3$ that splits $\{x\in R_1\,|\,\neg E(x,c_2)\}$. (In particular, $c_3\in R_1 \wedge \neg E(c_3, c_2)$.) Take any $c_4\in L_1$ that splits $L_1$. Then $E(c_3,c_4)\wedge\neg E(c_1,c_4)$. Since $c_3$ splits $\{x\in R_1\,|\,\neg E(x,c_2)\}\subseteq R$, it splits $R$ by definition. So $\neg E(d,c_1)\wedge\neg E(d, c_3)$. Thus $c_1$, $c_3$; $a_0$, $d$, $c_2$, $c_4$ witness that $E$ has VC-dimension $>1$.\\
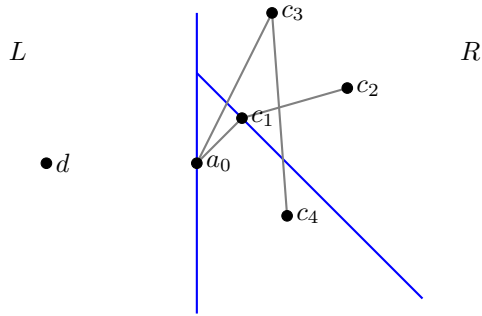
\begin{tikzpicture}
\draw[blue, thick] (0,2) -- (0,-2);
\draw[blue, thick] (3,-1.8) -- (0,1.2);
\draw[gray, thick] (1,2) -- (0,0);
\draw[gray, thick] (0.6,0.6) -- (0,0);
\draw[gray, thick] (0.6,0.6) -- (2,1);
\draw[gray, thick] (1,2) -- (1.2,-0.7);
\filldraw[black] (0,0) circle (2pt) node[anchor=west] {$a_0$};
\filldraw[black] (0.6,0.6) circle (2pt) node[anchor=west] {$c_1$};
\filldraw[black] (2,1) circle (2pt) node[anchor=west] {$c_2$};
\filldraw[black] (1,2) circle (2pt) node[anchor=west] {$c_3$};
\filldraw[black] (1.2,-0.7) circle (2pt) node[anchor=west] {$c_4$};
\filldraw[black] (-2,0) circle (2pt) node[anchor=west] {$d$};
\node[text width=3cm] at (-1,1.5){$L$};
\node[text width=3cm] at (5,1.5){$R$};
\end{tikzpicture}
\\
\indent
On the other hand, if $\exists x\in L_1$ splitting $L_1$, $\exists y\in R_1$, $\neg E(x,y)$, take $c_2\in L_1$, $c_3\in R_1$ such that $c_2$ splits $L_1$ and $\neg E(c_2,c_3)$. Take any $c_4\in L_1$ such that $E(c_2,c_4)$. Then $E(c_1,c_3)\wedge\neg E(c_2,c_3)\wedge E(c_2,c_4)\wedge\neg E(c_1,c_4)$. Since $c_2$ splits $L_1\subseteq R$, $c_2$ splits $R$ and hence $\neg E(d,c_2)$. So $c_1$, $c_2$; $a_0$, $d$, $c_3$, $c_4$ witness that $E$ has VC-dimension $>1$.
\\
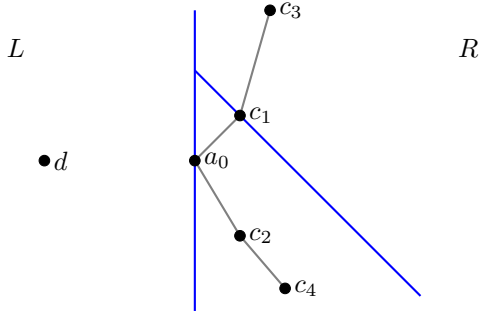
\begin{tikzpicture}
\draw[blue, thick] (0,2) -- (0,-2);
\draw[blue, thick] (3,-1.8) -- (0,1.2);
\draw[gray, thick] (0.6,-1) -- (0,0);
\draw[gray, thick] (0.6,0.6) -- (0,0);
\draw[gray, thick] (0.6,-1) -- (1.2,-1.7);
\draw[gray, thick] (1,2) -- (0.6,0.6);
\filldraw[black] (0,0) circle (2pt) node[anchor=west] {$a_0$};
\filldraw[black] (0.6,0.6) circle (2pt) node[anchor=west] {$c_1$};
\filldraw[black] (0.6,-1) circle (2pt) node[anchor=west] {$c_2$};
\filldraw[black] (1,2) circle (2pt) node[anchor=west] {$c_3$};
\filldraw[black] (1.2,-1.7) circle (2pt) node[anchor=west] {$c_4$};
\filldraw[black] (-2,0) circle (2pt) node[anchor=west] {$d$};
\node[text width=3cm] at (-1,1.5){$L$};
\node[text width=3cm] at (5,1.5){$R$};
\end{tikzpicture}
\\
\indent
So when $E$ is VC-dimension $1$, property $(*)$ must hold for some definable $A\subseteq V$ with positive $\delta$-dimension.
\end{proof}
\section{Revisiting stable case}\label{stabcase}
\begin{theorem}
 For each $k\in\mathbb{N}$,
 the family of $k$-stable graphs has the Erd\H{o}s-Hajnal property.    
\end{theorem}
\begin{proof}
    Fix $k\in\mathbb{N}$.
    Suppose no. 
    For each $i\in\omega$,
let $G_i=(V_i,E_i)$ be a finite 
$k$-stable graph 
such that
all homogeneous subsets of $G_i$ has size 
$<|V_i|^{\frac{1}{i}}$.
 Let $G=(V,E)$ be the ultraproduct
$\underset{i\in\omega}{\prod}V_i/
\mathcal{F}$. 
\\
\indent
    Let $A_{\emptyset}=V$. 
    By claim \ref{claim2},
    there is $a_\emptyset\in V$ such that 
    $\delta(E(a_\emptyset,V))>0$ and 
    $\delta(\neg
    E(a_\emptyset,V))>0$. 
    Suppose $\{a_\epsilon:
    \epsilon\in 2^{m}\}$
and    $\{A_{\epsilon}:
    \epsilon\in 2^m\}$ are defined where
    for each $\epsilon\in 2^m$,
    \begin{enumerate}
        \item 
        $a_\epsilon\in A_\epsilon$;
        \item 
        $\delta(E(a_\epsilon,A_\epsilon))>0$ and 
    $\delta(\neg
    E(a_\epsilon,A_\epsilon))>0$ (Hence $\delta(A_\epsilon)>0$).
    \end{enumerate}
    Take $A_{\epsilon^\frown 0}=
    \neg
    E(a_{\epsilon},
    A_{\epsilon})
    $,
    $A_{\epsilon^\frown 1}=
    E(a_{\epsilon},
    A_{\epsilon})$.
    By claim \ref{claim2},
    for any $\epsilon\in 2^{m+1}$,
    there is $a_{\epsilon}\in
    A_{\epsilon}$ such that
    $\delta(E(a_{\epsilon},
    A_{\epsilon}))>0$ and
    $\delta(\neg
    E(a_{\epsilon},
    A_{\epsilon}))>0$.
    Then $\{\underset{\epsilon
    \prec p}{\bigcap}
    A_{\epsilon}
    :
    p\in 2^{\omega}\}$ is a collection 
    of $2^\omega$ many $E$-types with parameters in the countable set
    $\{a_\epsilon: 
    \epsilon\in 2^{<\omega}\}$. 
    By fact \ref{stabequiv},
    $E$ is not $k$-stable, 
    a contradiction.
    \\
    \indent
    (Note: We assume here $E$ to be a binary relation. The author doesn't know how to avoid using Shelah's 2-rank for hypergraphs.)
\end{proof}
\section{VC-dimension 2}\label{vcq2}
\indent
\begin{theorem}
    The family of graphs with VC-dimension $\leq 2$ has the Erd\H{o}s-Hajnal property.
\end{theorem}
\begin{proof}
Proof of substitution combined with 
Erd\H os-Hajnal property for $\{C_6,\overline{C_6}\}$-free graphs (fact \ref{c6eh}) 
gives Erd\H os-Hajnal property for VC-dimension 2.
\\
\indent
By fact \ref{c6eh}, fix $c>0$ such that for any $\{C_6,\overline{C_6}\}$-free graph $P$ ,
$P$ has a homogeneous subset of size $\geq|P|^c$. 
Let $\delta$ satisfy $\frac{1}{2}-6\delta>0$, $c\delta<\frac{1}{2}-6\delta$, 
$G$ be a graph with $|G|=n$ such that the largest size of a homogeneous set of $G$ is $<|G|^{c\delta}$, 
and $m=\lceil n^\delta\rceil>6$.
Then $G$ has at least $\dfrac{\binom{n}{m}}{\binom{n-6}{m-6}}$ induced subgraphs isomorphic to $C_6$ or $\overline{C_6}$.
Then there are at least $\dfrac{\binom{n}{m}}{2\binom{n-6}{m-6}}$ copies of $C_6$ or 
$\dfrac{\binom{n}{m}}{2\binom{n-6}{m-6}}$ copies of $\overline{C_6}$.
Replacing $G$ with $\overline{G}$ if necessary,
may assume the first case holds. 
We can find $u_1,u_3,u_4,u_5,u_6$ that are the first, third, forth, fifth, sixth points on the 
cycle respectively, for $\dfrac{\binom{n}{m}}{2n(n-1)(n-2)(n-3)(n-4)\binom{n-6}{m-6}}$ many induced $6$-cycles. 
Among these copies the size of the set of the second 
point on the cycle is at least $\dfrac{\binom{n}{m}}{2n(n-1)(n-2)(n-3)(n-4)\binom{n-6}{m-6}}=\dfrac{n-5}{2m...(m-5)}$. 
So we will have the family of graphs not inducing $C_6$ with a vertex substituted by an edge or 
$\overline{C_6}$ with a vertex substituted by a pair of nonadjacent vertices satisfies the Erd\H os-Hajnal property.
Repeat this argument and we will replace the forth vertex on the cycle by an edge and then the sixth vertex.
We will then get the following graph:\\\\\\
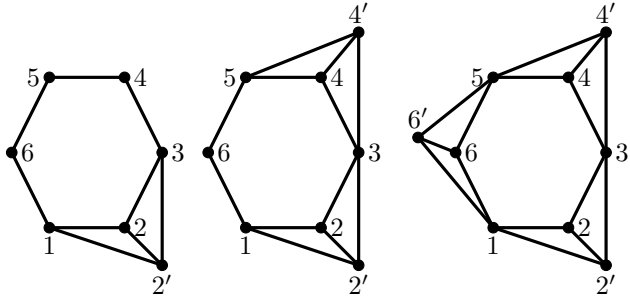
\begin{tikzpicture}
\draw[black, very thick] (0,0)--(1,0)--(1.5,1)--(1,2)--(0,2)--(-0.5,1)--cycle;
\filldraw[black] (0,0) circle (2pt) node[anchor=north] {$1$};
\filldraw[black] (1,0) circle (2pt) node[anchor=west] {$2$};
\filldraw[black] (1.5,1) circle (2pt) node[anchor=west] {$3$};
\filldraw[black] (1,2) circle (2pt) node[anchor=west] {$4$};
\filldraw[black] (0,2) circle (2pt) node[anchor=east] {$5$};
\filldraw[black] (-0.5,1) circle (2pt) node[anchor=west] {$6$};
\draw[black, very thick] (1,0) -- (1.5,-0.5);
\filldraw[black] (1.5,-0.5) circle (2pt) node[anchor=north] {$2'$};
\draw[black, very thick] (0,0) -- (1.5,-0.5);
\draw[black, very thick] (1.5,1) -- (1.5,-0.5);
\end{tikzpicture}
\begin{tikzpicture}
\draw[black, very thick] (0,0)--(1,0)--(1.5,1)--(1,2)--(0,2)--(-0.5,1)--cycle;
\filldraw[black] (0,0) circle (2pt) node[anchor=north] {$1$};
\filldraw[black] (1,0) circle (2pt) node[anchor=west] {$2$};
\filldraw[black] (1.5,1) circle (2pt) node[anchor=west] {$3$};
\filldraw[black] (1,2) circle (2pt) node[anchor=west] {$4$};
\filldraw[black] (0,2) circle (2pt) node[anchor=east] {$5$};
\filldraw[black] (-0.5,1) circle (2pt) node[anchor=west] {$6$};
\draw[black, very thick] (1,0) -- (1.5,-0.5);
\filldraw[black] (1.5,-0.5) circle (2pt) node[anchor=north] {$2'$};
\draw[black, very thick] (0,0) -- (1.5,-0.5);
\draw[black, very thick] (1.5,1) -- (1.5,-0.5);
\filldraw[black] (1.5,2.6) circle (2pt) node[anchor=south] {$4'$};
\draw[black, very thick] (1,2) -- (1.5,2.6);
\draw[black, very thick] (1.5,1) -- (1.5,2.6);
\draw[black, very thick] (0,2) -- (1.5,2.6);
\end{tikzpicture}
\begin{tikzpicture}
\draw[black, very thick] (0,0)--(1,0)--(1.5,1)--(1,2)--(0,2)--(-0.5,1)--cycle;
\filldraw[black] (0,0) circle (2pt) node[anchor=north] {$1$};
\filldraw[black] (1,0) circle (2pt) node[anchor=west] {$2$};
\filldraw[black] (1.5,1) circle (2pt) node[anchor=west] {$3$};
\filldraw[black] (1,2) circle (2pt) node[anchor=west] {$4$};
\filldraw[black] (0,2) circle (2pt) node[anchor=east] {$5$};
\filldraw[black] (-0.5,1) circle (2pt) node[anchor=west] {$6$};
\draw[black, very thick] (1,0) -- (1.5,-0.5);
\filldraw[black] (1.5,-0.5) circle (2pt) node[anchor=north] {$2'$};
\draw[black, very thick] (0,0) -- (1.5,-0.5);
\draw[black, very thick] (1.5,1) -- (1.5,-0.5);
\filldraw[black] (1.5,2.6) circle (2pt) node[anchor=south] {$4'$};
\draw[black, very thick] (1,2) -- (1.5,2.6);
\draw[black, very thick] (1.5,1) -- (1.5,2.6);
\draw[black, very thick] (0,2) -- (1.5,2.6);
\filldraw[black] (-1,1.2) circle (2pt) node[anchor=south] {$6'$};
\draw[black, very thick] (-1,1.2) -- (-0.5,1);
\draw[black, very thick] (-1,1.2) -- (0,2);
\draw[black, very thick] (-1,1.2) -- (0,0);
\end{tikzpicture}
\\
(The edge relation between $2'$, $4'$ and $6'$ doesn't matter.)\\\\\\
\indent
Suppose Erd\H os-Hajnal property fails for the family of finite graphs with VC-dimension $2$.
 For each $i\in\omega$,
let $G_i=(V_i,E_i)$ be a finite 
graphs with VC-dimension $2$
such that
all homogeneous subsets of $G_i$ has size 
$<|V_i|^{\frac{1}{i}}$.
 Let $G=(V,E)$ be the ultraproduct
$\underset{i\in\omega}{\prod}V_i/
\mathcal{F}$. 
Then there is $x\in V$ such that $\delta(E(x,V))>0$ and there is $y\in E(x,V)$ such that $\delta(\neg E(y,V)\cap E(x,V))>\alpha>0$.
Now consider the definable sets $W=\neg E(y,V)\cap E(x,V)$ such that $|W_i|>|V_i|^\alpha$ for all $i\in F$, some $F\in \mathcal{F}$. 
May assume $i$ is large. 
By the above, there is in $W_i$ or in the complement of $W_i$ an induced copy of $6$-cycle with the 
second, forth, sixth points replaced by an edge respectively. 
Thus $(V,E)$ has VC-dimension $>2$, a contradiction.\\
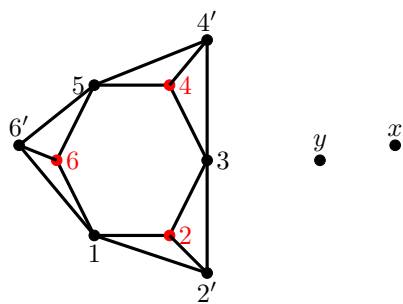
\begin{tikzpicture}
\draw[black, very thick] (0,0)--(1,0)--(1.5,1)--(1,2)--(0,2)--(-0.5,1)--cycle;
\filldraw[black] (0,0) circle (2pt) node[anchor=north] {$1$};
\filldraw[red] (1,0) circle (2pt) node[anchor=west] {$2$};
\filldraw[black] (1.5,1) circle (2pt) node[anchor=west] {$3$};
\filldraw[red] (1,2) circle (2pt) node[anchor=west] {$4$};
\filldraw[black] (0,2) circle (2pt) node[anchor=east] {$5$};
\filldraw[red] (-0.5,1) circle (2pt) node[anchor=west] {$6$};
\draw[black, very thick] (1,0) -- (1.5,-0.5);
\filldraw[black] (1.5,-0.5) circle (2pt) node[anchor=north] {$2'$};
\draw[black, very thick] (0,0) -- (1.5,-0.5);
\draw[black, very thick] (1.5,1) -- (1.5,-0.5);
\filldraw[black] (1.5,2.6) circle (2pt) node[anchor=south] {$4'$};
\draw[black, very thick] (1,2) -- (1.5,2.6);
\draw[black, very thick] (1.5,1) -- (1.5,2.6);
\draw[black, very thick] (0,2) -- (1.5,2.6);
\filldraw[black] (-1,1.2) circle (2pt) node[anchor=south] {$6'$};
\draw[black, very thick] (-1,1.2) -- (-0.5,1);
\draw[black, very thick] (-1,1.2) -- (0,2);
\draw[black, very thick] (-1,1.2) -- (0,0);
\filldraw[black] (3,1) circle (2pt) node[anchor=south] {$y$};
\filldraw[black] (4,1.2) circle (2pt) node[anchor=south] {$x$};
\end{tikzpicture}
\end{proof}
\bibliographystyle{alpha}
\bibliography{ref}

\begin{thebibliography}{CSSS23}

\bibitem[APS01]{alon2001ramsey}
Noga Alon, J{\'a}nos Pach, and J{\'o}zsef Solymosi.
\newblock Ramsey-type theorems with forbidden subgraphs.
\newblock {\em Combinatorica}, 21(2):155--170, 2001.

\bibitem[BB18]{bays2018projective}
Martin Bays and Emmanuel Breuillard.
\newblock Projective geometries arising from elekes-szab$\backslash$'o
  problems.
\newblock {\em arXiv preprint arXiv:1806.03422}, 2018.

\bibitem[CS18a]{chernikov2018note}
Artem Chernikov and Sergei Starchenko.
\newblock A note on the erd{\H{o}}s-hajnal property for stable graphs.
\newblock {\em Proceedings of the American Mathematical Society},
  146(2):785--790, 2018.

\bibitem[CS18b]{chernikov2018regularity}
Artem Chernikov and Sergei Starchenko.
\newblock Regularity lemma for distal structures.
\newblock {\em Journal of the European Mathematical Society},
  20(10):2437--2466, 2018.

\bibitem[CSSS23]{chudnovsky2023erdHos}
Maria Chudnovsky, Alex Scott, Paul Seymour, and Sophie Spirkl.
\newblock Erd{\H{o}}s--hajnal for graphs with no 5-hole.
\newblock {\em Proceedings of the London Mathematical Society},
  126(3):997--1014, 2023.

\bibitem[EH89]{erdos1989ramsey}
Paul Erd{\"o}s and Andr{\'a}s Hajnal.
\newblock Ramsey-type theorems.
\newblock {\em Discrete Applied Mathematics}, 25(1-2):37--52, 1989.

\bibitem[FPS19]{fox2019erdHos}
Jacob Fox, J{\'a}nos Pach, and Andrew Suk.
\newblock Erd{\H{o}}s--hajnal conjecture for graphs with bounded vc-dimension.
\newblock {\em Discrete \& Computational Geometry}, 61:809--829, 2019.

\bibitem[MS14]{malliaris2014regularity}
Maryanthe Malliaris and Saharon Shelah.
\newblock Regularity lemmas for stable graphs.
\newblock {\em Transactions of the American Mathematical Society},
  366(3):1551--1585, 2014.

\bibitem[She90]{shelah1990classification}
Saharon Shelah.
\newblock {\em Classification theory: and the number of non-isomorphic models}.
\newblock Elsevier, 1990.

\bibitem[Sim15]{simon2015guide}
Pierre Simon.
\newblock {\em A guide to NIP theories}.
\newblock Cambridge University Press, 2015.

\end{thebibliography}
\end{document}